\documentclass[a4paper, 12pt]{article}
\usepackage{amsmath, amsthm}
\usepackage{amsfonts}
\usepackage{amssymb}
\usepackage[square,sort,comma,numbers]{natbib}
\usepackage{fullpage}
\usepackage{todonotes}

\usepackage{graphicx}

\newcommand{\Sfrac}[2]{{ \textstyle \frac{#1}{#2}}}

\newtheorem{thm}{Theorem} [section]

\newtheorem{lem}[thm]{Lemma}
\newtheorem{prop}[thm]{Proposition}
\newtheorem{df}[thm]{Definition}

\newcommand{\intall}{\int_{\R}}
\newcommand{\inthalf}{\int_{\R_+}}

\def\R{I\!\!R}

\def\pa{\partial}

\def\cal{\mathcal}

\begin{document}
$\;$

\vspace{0.4in}

\centerline{\Large{Existence and Uniqueness of Singular Solutions for a}} 

\vspace{0.1in}

\centerline{\Large{Conservation Law Arising in Magnetohydrodynamics}}


\vspace{0.5in}

\centerline{\large{Henrik Kalisch}\footnote{\texttt{Department of Mathematics, University of Bergen, Bergen, Norway}
            },
            \large{Darko Mitrovic}\footnote{\texttt{Department of Mathematics, University of Montenegro, Podgorica, Montenegro}
            \break
            }, \large{Vincent Teyekpiti$^{1}$}}

\vspace{.4in}

\begin{abstract}
The Brio system is a two-by-two system of conservation laws arising as a simplified 
model in ideal magnetohydrodynamics (MHD). The system has the form
\begin{align*}
\partial_tu+\partial_x\big(\Sfrac{u^2+v^2}{2}\big)=0,\\
\partial_tv+\partial_x\big(v(u-1)\big)=0.
\end{align*} 
It was found in previous works that the standard theory of hyperbolic conservation laws
does not apply to this system since the characteristic fields are not genuinely nonlinear
on the set $v=0$. As a consequence, certain Riemann problems have no weak solutions
in the traditional class of functions of bounded variation.

It was argued in \cite{HL} that in order to solve the system, singular solutions
containing Dirac masses along the shock waves might have to be used. Solutions
of this type were exhibited in \cite{KM,Sarr}, but uniqueness was not obtained.

In the current work, we introduce a nonlinear change of variables which makes it possible
to solve the Riemann problem in the framework of the standard theory of conservation laws.
In addition, we develop a criterion which leads to an admissibility condition
for singular solutions of the original system, 
and it can be shown that admissible solutions are unique in the framework developed here.
\end{abstract}

\vspace{.3in}


\section{Introduction}

Conservation laws have been used as a mathematical tool in a variety of
situations in order to provide a simplified description of complex
physical phenomena which nevertheless keeps the essential features
of the processes to be described, and the general theory of hyperbolic conservation
laws aims to provide a unified set of techniques needed to understand
the mathematical properties of such equations. However, in some cases,
the general theory fails to provide a firm mathematical description 
for a particular case because some of the assumptions needed
in the theory are not in place.

In the present contribution we focus on such an example,
a hyperbolic conservation law appearing in ideal magnetohydrodynamics. 
For this conservation law, solutions cannot be found using the classical
techniques of conservation laws, and a new approach is needed.

Magnetohydrodynamics (MHD) is the study of how electric currents in a moving
conductive fluid interact with the magnetic field created by the moving fluid itself.
The MHD equations are a combination of the Navier-Stokes equations of fluid mechanics
and Maxwell's equations of electromagnetism, and the equations are generally coupled
in such a way that they must be solved simultaneously. The ideal MHD equations are
based on a combination of the Euler equations of fluid mechanics 
(i.e. for an inviscid and incompressible fluid)
and a simplified form of Maxwell's equations.
The resulting system is highly complex and one needs to rely on numerical 
approximation of solutions in order to understand the dynamics of the system.

As even the numerical study of the full system is very challenging, it can be
convenient to introduce some simplifying assumptions -- valid in some limiting cases --
in order to get a better idea of the qualitative properties of the system, 
and in order to provide some test cases against which numerical codes 
for the full MHD system can be tested.

The emergence of coherent structures in turbulent plasmas has
been long observed both in numerical simulations and experiments.
%
%
Moreover, the tendency of the magnetic field to organize into
low-dimensional structures such as two-dimensional
magnetic pancakes and one-dimensional magnetic ropes is well known.
As a consequence, in certain cases it makes sense to use simplified
one or two dimensional model equations.
Such simplified equations will be easier to solve, but nevertheless preserve some
of the important features observed in MHD systems.
In \cite{brio}, a simplified model system for 
ideal MHD was built using such phenomenological considerations. The system is written as
\begin{equation}\label{Brio System}
\begin{split}
\partial_t u+\partial_x\big(\Sfrac{u^2+v^2}{2}\big)=0,\\
\partial_t v+\partial_x\big(v(u-1)\big)=0.
\end{split}
\end{equation} 
The quantities $u$ and $v$ are the velocity components of the fluid whose 
dynamics is determined by MHD forces, and the system represents the conservation of the velocities. 
Velocity conservation in this form holds only in idealized situations in the case of smooth solutions, 
and 
the limitation of this assumption manifests itself in the non-solvability of the system 
even for the simplest piece-wise constant initial data, i.e. 
for certain dispositions of the Riemann initial data
\begin{equation}
\label{riemann}
u|_{t=0}=\begin{cases}
U_L, &x\leq 0\\
U_R, &x> 0
\end{cases}, \qquad v|_{t=0}=\begin{cases}
V_L, &x\leq 0\\
V_R, &x> 0
\end{cases}.
\end{equation}

From a mathematical point of view, the characteristic fields of this system are neither 
genuinely nonlinear nor linearly degenerate in certain regions in the $(u,v)$-plane (see \cite{HL}). 
In this case the standard theory of hyperbolic conservation laws which can be found in e.g. \cite{Daf} 
does not apply and one cannot find a classical Riemann solution admissible in the sense of 
Lax \cite{Lax} or Liu \cite{Liu}.

In order to deal with the problem of non-existence of solutions to the
Riemann problem for certain conservation laws, the concept of singular solutions
incorporating $\delta$-distributions along shock trajectories was introduced in \cite{Korchinski}. 
The idea was pursued further in \cite{KK,HL}, and by now, the literature on the subject is rather
extensive. Some authors have defined theories of distribution products in order to incorporate
the $\delta$-distributions into the notion of weak solutions \cite{DMLM,huangWang,Sarr}.
In other works, the need to multiply $\delta$-distributions has been avoided either
by working with integrated equations \cite{Huang,KaTe}, or by making an appropriate definition
of singular solutions \cite{DSH}. In order to find admissibility conditions for such
singular solutions, some authors have used the weak asymptotic method \cite{DSH,DOS,MN_arma,Omelyanov}.
With the aim of dealing with the nonlinearity featured by the system \eqref{Brio System}, 
the weak asymptotic method was also extended to include complex-valued approximations \cite{KM}.
The authors of \cite{KM} were able to provide singular solutions of \eqref{Brio System}
even in cases which could not be resolved earlier. However, even if \cite{KM} provides some
admissibility conditions, the authors of \cite{KM} did not succeed to prove uniqueness.
Existence of singular solutions to \eqref{Brio System} was also proved in \cite{Sarr} using
the theory of distribution products, but uniqueness could not be obtained.

Therefore, it was natural to ask whether the Brio system should be solved 
in the framework of $\delta$-distributions 
as conjectured in \cite{HL} where the system was first considered 
from the viewpoint of the conservation laws theory. 
The authors of \cite{HL} compared \eqref{Brio System} with the triangular system
\begin{equation}
\label{triang}
\begin{split}
\partial_t u+\partial_x\big(\Sfrac{u^2}{2}\big)=0,\\
\partial_t v+\partial_x\big(v(u-1)\big)=0.
\end{split}
\end{equation}  
which differs from \eqref{Brio System} in the quadratic term $v^2$. However, the system \eqref{triang} is linear with respect to $v$ and it naturally admits $\delta$-type solutions (obtained e.g. via the vanishing viscosity approximation). To this end, let us remark that most of the systems admitting $\delta$-shock wave solutions are linear with respect to one of the unknown functions \cite{DMLM, DSH, HL, huangWang, KK}. There are also a number of systems which can be solved only by introducing the $\delta$-solution concept and which are non-linear with respect to both of the variables such as the chromatography system \cite{Sun} or the Chaplygin gas system \cite{N1}. However, in all such systems, it was possible to control the nonlinear operation over an approximation of the $\delta$-distribution. This is not the case with \eqref{Brio System} since the term $u^2+v^2$ will necessarily tend to infinity for any real approximation of the $\delta$-function.
This problem can be dealt with by introducing complex-valued approximations of the $\delta$-distribution. 
Using this approach, a somewhat general theory can be developed as follows.
Consider the system
\begin{equation}
\label{gensystem}
\begin{split}
\pa_t u + \pa_x f(u,v)  =&  0,       \\
\pa_t v + \pa_x g(u,v) =& 0,
\end{split}
\end{equation} 
The following definition gives the notion of $\delta$-shock solution to system \eqref{gensystem}.
\begin{df}
\label{def-gen}
The pair of distributions
\begin{equation}
\label{delta-sol}
u=U+\alpha(x,t)\delta(\Gamma) \ \ \text{and} \ \ v=V+\beta(x,t)\delta(\Gamma)
\end{equation} are called a generalized $\delta$-shock wave solution of system
\eqref{gensystem} with the initial data
$U_0(x)$ and $V_0(x)$
if the integral identities
\begin{align}
\label{m1-g1}
\nonumber 
&\inthalf \! \! \intall \left( U\pa_t\varphi +f(U,V)
\pa_x\varphi\right)~dxdt \\ 
&\qquad + \sum\limits_{i\in I}\int_{\gamma_i}\alpha_i(x,t)
\Sfrac{\pa \varphi(x,t)}{\pa {\bf l}} \,  + \intall U_0(x)\varphi(x,0)~dx = 0,
\end{align}
\begin{align}
\label{m2-g1}
\nonumber
&\inthalf \! \! \intall \left(V\pa_t\varphi+ g(U,V)\pa_x\varphi\right)~dxdt\\
&\qquad + \sum\limits_{i\in I}\int_{\gamma_i}\beta_i(x,t)
\Sfrac{\pa \varphi(x,t)}{\pa {\bf l}} \, + \int_{\R} V_0(x)
\varphi(x,0)~dx  = 0, 
\end{align}
hold for all test functions $
\varphi\in {\cal D}(\R\times \R_+)$.
\end{df}

This definition may be interpreted as an extension of the
classical notion of weak solutions. 
The definition is consistent with the concept of measure solutions 
as put forward in \cite{DMLM, huangWang} in the sense that the two singular parts 
of the solution coincide, while the regular parts differ on a set of Lebesgue measure zero. 
However, Definition \ref{def-gen} can be applied to any hyperbolic system of equations 
while the solution concept from \cite{DMLM} only works in the special situation 
when the $\delta$-distribution is attached to an unknown which appears linearly in the flux $f$ or $g$,
or when nonlinear operations on $\delta$ can somehow be controlled in another way.

Definition \ref{def-gen} is quite general, allowing a combination of initial
steps and delta distributions; but its effectiveness is already
demonstrated by considering the Riemann problem with a single jump.
Indeed, for this configuration it can be shown that a $\delta$-shock
wave solution exists for any $2\times 2$ system of conservation laws.

Consider the Riemann problem for \eqref{gensystem} with initial data
$u(x,0) = U_0(x)$ and  $v(x,0) = V_0(x)$,
where
\begin{equation}
\label{rieman-data}
U_0(x)=\begin{cases} u_1, &x<0,\\
u_2, & x>0,
\end{cases} \quad
V_0(x)=\begin{cases} v_1, &x<0,\\
v_2, & x>0.
\end{cases}
\end{equation}
Then, the following theorem holds:
\begin{thm}\label{thm-cnl}
{\bf a)}
If $u_1\neq u_2$ then the pair of distributions
\begin{eqnarray}\label{sol-a1}
u(x,t) & = & U_0(x-ct), \\
\label{sol-a2}
v(x,t) & = & V_0(x-ct) + \beta(t) \delta(x-ct),
\end{eqnarray}
where
\begin{equation}
\label{RH-def1}
c=\frac{[f(U,V)]}{[U]}=\frac{f(u_2,v_2)-f(u_1,v_1)}{u_2-u_1}, \ \mbox{ and } \
\beta(t)=(c[V]-[g(U,V)])t,
\end{equation}
represents the $\delta$-shock wave solution of \eqref{gensystem}
with initial data $U_0(x)$ and $V_0(x)$
in the sense of Definition \ref{def-gen} with $\alpha(t)=0$.

\vskip 0.1in

\noindent
{\bf b)} If $v_1\neq v_2$ then the pair of distributions
\begin{eqnarray}\label{sol-b1}
u(x,t) & = & U_0(x-ct) + \alpha(t) \delta(x-ct), \\
\label{sol-b2}
v(x,t) & = & V_0(x-ct),
\end{eqnarray}
where
\begin{equation}
\label{RH-def2}
c=\frac{[g(U,V)]}{[V]}=\frac{g(u_2,v_2)-g(u_1,v_1)}{v_2-v_1}, \ \
\alpha(t)=(c[U]-[f(U,V)])t
\end{equation}
represents the $\delta$-shock solution of \eqref{gensystem} with
initial data $U_0(x)$ and $V_0(x)$ in the sense of Definition
\ref{def-gen} with $\beta(t)=0$.
\end{thm}
\begin{proof}
We will prove only the first part of the theorem as the
second part can be proved analogously. We immediately see that
$u$ and $v$ given by \eqref{sol-a1} and \eqref{sol-a2}
satisfy \eqref{m1-g1} since $c$ is given exactly by the
Rankine-Hugoniot condition derived from that system. By
substituting $u$ and $v$ into \eqref{m2-g1},
we get after standard transformations:
\begin{align*}
\int_{\R_+}\left(-c[V]+[g(u,V)]\right)\varphi(ct,t)~dt
-\int_{\R_+}\alpha'(t)\varphi(ct,t)~dt = 0.
\end{align*}
From here and since $\alpha(0)=0$, the conclusion follows
immediately.
\end{proof}

As the solution framework of Definition \ref{def-gen} is very weak,
one might expect non-uniqueness issues to arise.
This is indeed the case, and the proof of the following proposition
is an easy exercise.
\begin{prop}\label{non-unique}
System \eqref{gensystem} with the zero initial data:
$u|_{t=0}=v|_{t=0}=0$ admits $\delta$-shock solutions of the form:
\begin{align*}
u(x,t)=0, \ \ v(x,t)=\beta\delta(x-c_1t)-\beta\delta(x-c_2 t),
\end{align*}for arbitrary constants $\beta$, $c_1$ and $c_2$.
\end{prop}

As already alluded to, a different formal approach for solving \eqref{Brio System} 
was used by \cite{Sarr}. However, just as in \cite{KM} the definition of singular solutions 
used in \cite{Sarr} is so weak that uniqueness cannot be obtained. 
Another problem left open in \cite{KM, Sarr} is the physical meaning of the $\delta$-distribution 
appearing as the part of the solution. 
Considering systems such as the Chaplygin gas system or \eqref{triang}, 
the use of the $\delta$-distribution in the solution can be justified by 
invoking extreme concentration effects if we assume that $v$ represents density. 
However, in the case of the Brio system, $u$ and $v$ are velocities and unbounded
velocities cannot be explained in any reasonable physical way.

In the present contribution, we shall try to explain necessity of $\delta$-type 
solutions for \eqref{Brio System} following considerations from \cite{KT} where it was argued 
(in a quite different setting) that the wrong variables are conserved. 
In other words, the presence of a $\delta$-distribution in a weak solution actually signifies 
the inadequacy of the corresponding conservation law in the case of weak solutions.
Similar consideration were recently put forward in the case of singular solutions in
the shallow-water system \cite{KMT}.

Starting from this point, we are able to formulate uniqueness requirement for the Riemann problem 
for \eqref{Brio System}. First, we shall rewrite the system using the energy $q=(u^2+v^2)/2$ 
as one of the conserved quantities (which is actually an entropy function corresponding to \eqref{Brio System}). 
Thus, we obtain a strictly hyperbolic and genuinely nonlinear system which admits a Lax admissible solution 
for any Riemann problem. Such a solution is unique and it will give a unique $\delta$-type solution 
to the original system. The $\delta$-distribution will necessarily appear 
due to the nonlinear transformation that we apply.

The paper is organized as follows: In Section 2, we shall rewrite \eqref{Brio System} 
in the new variables $q$ and $u$, and exhibit the admissible shock and rarefaction waves.
In Section 3, we shall introduce the admissibility concept for solutions of the
original system \eqref{Brio System},
and prove existence and uniqueness of a solution to the Riemann problem
in the framework of that definition. 

\section{Energy-velocity conservation}
As mentioned above, conservation of velocity is not necessarily a physically well defined balance law,
and it might be preferable to specify conservation of energy for example.
Actually, in some cases, conservation of velocity does give an appropriate balance law,
such as for example in the case of shallow-water flows \cite{GKK}.
In the present situation, it appears natural to replace at least one of the
equations of velocity conservation. As will be seen momentarily,
such a system will be strictly 
hyperbolic with genuinely nonlinear characteristic fields,
so that the system will be more amenable to standard method of hyperbolic conservation laws.
To introduce the new conservation law, we define an energy function 
\begin{equation}\label{Entropy Function}
q(u,v)=\frac{u^2+v^2}{2},
\end{equation}
and note that this function is a mathematical entropy for the system \eqref{Brio System}.
Then we use the transformation
\begin{equation*}
(u,v)\rightarrow \big(u,\Sfrac{u^2+v^2}{2} \big),
\end{equation*}
to transform \eqref{Brio System} into the system 
\begin{equation}\label{Transformed system}
\begin{split}
\partial_tu + \partial_xq=0,\\
\partial_tq + \partial_x \big((2u-1)q+\Sfrac{u^2}{2}-\Sfrac{2u^3}{3} \big)=0.
\end{split}
\end{equation} 

System \eqref{Brio System} and the transformed  system \eqref{Transformed system}
are equivalent for differentiable solutions. 
However, as will be evident momentarily, the nonlinear transformation changes the character of the system, 
and while \eqref{Brio System} is not always genuinely nonlinear, 
the new system \eqref{Transformed system} is always strictly hyperbolic and  genuinely nonlinear.

In the following, we analyze \eqref{Transformed system}, and find the elementary
waves for the solution of \eqref{Transformed system}.
The flux function of the new system is given by 
\begin{equation*}
F=
	\begin{pmatrix}
	q\\ (2u-1)q+\frac{u^2}{2}-\frac{2u^3}{3}
	\end{pmatrix}
\end{equation*}
with flux Jacobian 
\begin{equation*}
DF=
	\begin{pmatrix}
	0 & 1\\ 2q+u-2u^2 & 2 u-1
	\end{pmatrix}.
\end{equation*}
The characteristic velocities are given by
\begin{equation}\label{eigenvalues}
\lambda_{-,+} = \frac{2u-1 \mp \sqrt{8q-4u^2+1}}{2}.
\end{equation}
A direct consequence of \eqref{Entropy Function} gives the 
relation $2q \geq u^2 \geq 0$ which implies that the quantity under the square 
root is non-negative. Thus, $8q-4u^2+1>0$ and the eigenvalues are real and 
distinct so that the system is strictly hyperbolic.
The right eigenvectors in this case are given by 
\begin{equation}
\begin{split}
r_- = 
\begin{pmatrix}
1 \\ u-\frac{1}{2} - \sqrt{2q-u^2+\frac{1}{4}}
\end{pmatrix},\\
r_+ = 
\begin{pmatrix}
1 \\ u-\frac{1}{2} + \sqrt{2q-u^2+\frac{1}{4}}
\end{pmatrix}.
\end{split}
\end{equation}
It can be verified easily that these eigenvectors are linearly independent 
and span the $(u,q)$-plane.
The associated characteristic fields 
\begin{equation}
\nabla\lambda_-\cdot r_- = 2+\frac{1}{\sqrt{8q-4u^2+1}} \ ,
\end{equation}
\begin{equation}
\nabla\lambda_+\cdot r_+ = 2-\frac{1}{\sqrt{8q-4u^2+1}} \ ,
\end{equation}
are genuinely nonlinear and admit both shock and rarefaction waves.
For a shock profile connecting a constant left state $(u,q)=(u_L,q_L)$ to a constant right 
state $(u,q)=(u_R,q_R)$, the Rankine-Hugoniot jump conditions for \eqref{Transformed system} are 
\begin{align}\label{Rankine-Hugoniot}
c(u_L-u_R)&=~(q_L-q_R),\\
\label{Rankine-Hugoniot1}
c(q_L-q_R)&=\big( (2u_L-1)q_L+\Sfrac{u_L^2}{2}-\Sfrac{2u_L^3}{3} - 
(2u_R-1)q_R-\Sfrac{u_R^2}{2}+\Sfrac{2u_R^3}{3} \big),
\end{align}
where $c$ is the shock speed. We want the speed 
in \eqref{Rankine-Hugoniot}, \eqref{Rankine-Hugoniot1} 
to satisfy the Lax admissibility condition
\begin{equation}
\label{Lax}
\lambda_\mp(u_L,q_L)\geq c \geq \lambda_\mp(u_R,q_R).
\end{equation}
To determine the set of all states that can be connected to a 
fixed left state $(u_L,q_l)$, we eliminate the shock speed, $c$, from the above equations 
to obtain the shock curves
\begin{small}
\begin{align*}
&(q_R)_{1,2}=\frac{2q_L-(u_L-u_R)(2u_R-1)}{2} \ \ \pm \ \ \ \ \ \ \ \ \ \ \ \ \ \ \ \ \ \ \ \ \ \ \ \ \ \ \ \ \ \ \ \ \ \ \ \ \ \ \ \ \ \ \ \ \ \ \ \ \ \ \\ 
 &\frac{\sqrt{[-2q_L+(u_L-u_R)(2u_R-1)]^2+4\left[(u_L-u_R)\left((2u_L-1)q_L+\frac{u_L^2}{2}-
\frac{u_R^2}{2}-\frac{2u_L^3}{3}+\frac{2u_R^3}{3}\right)-q_L^2\right]}}{2}.
\end{align*}
\end{small}
After basic algebraic manipulations, we obtain
\begin{align}\label{SW}
\nonumber
(q_R)_{1,2}= q_L&-\frac{1}{2}(u_L-u_R)(2u_R-1)\\ 
&\pm \mid u_L-u_R\mid\sqrt{2q_L + \Sfrac{1}{4} + \Sfrac{1}{2} (u_L-u_R)-\Sfrac{1}{3} \big( 2u_L^2 +2u_Lu_R-u_R^2\big)}
\end{align}
From here and \eqref{Lax}, by considering $(u_R,q_R)$ in a small neighborhood of $(u_L,q_L)$, we conclude that the shock wave of the first family (SW1), the shock wave of the second family (SW2), the rarefaction wave of the first family (RW1) and the rarefaction wave of the second family (RW2) are given as follows:
\begin{figure}
  \begin{center}
     {\includegraphics[scale=0.5]{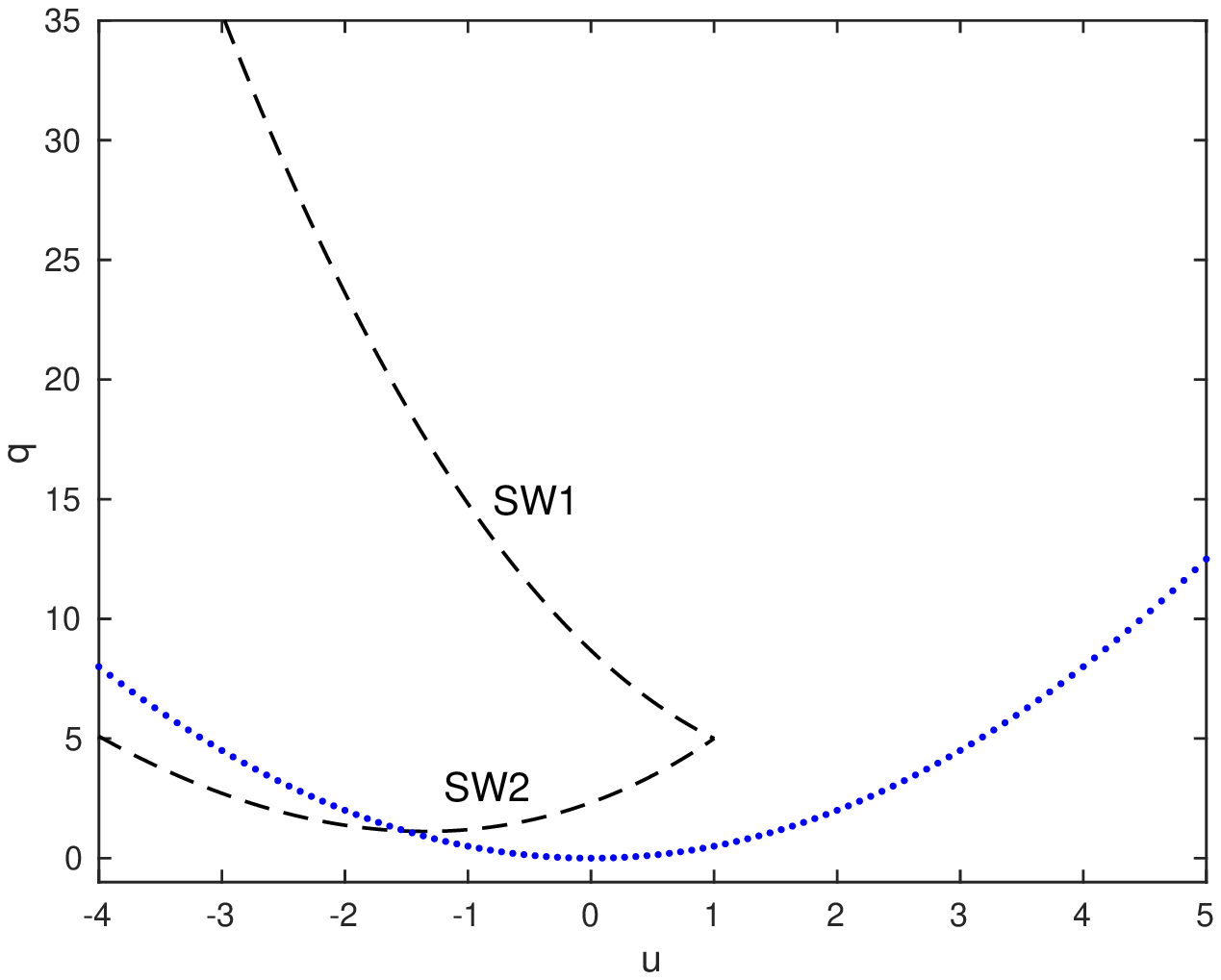}}~~
     {\includegraphics[scale=0.5]{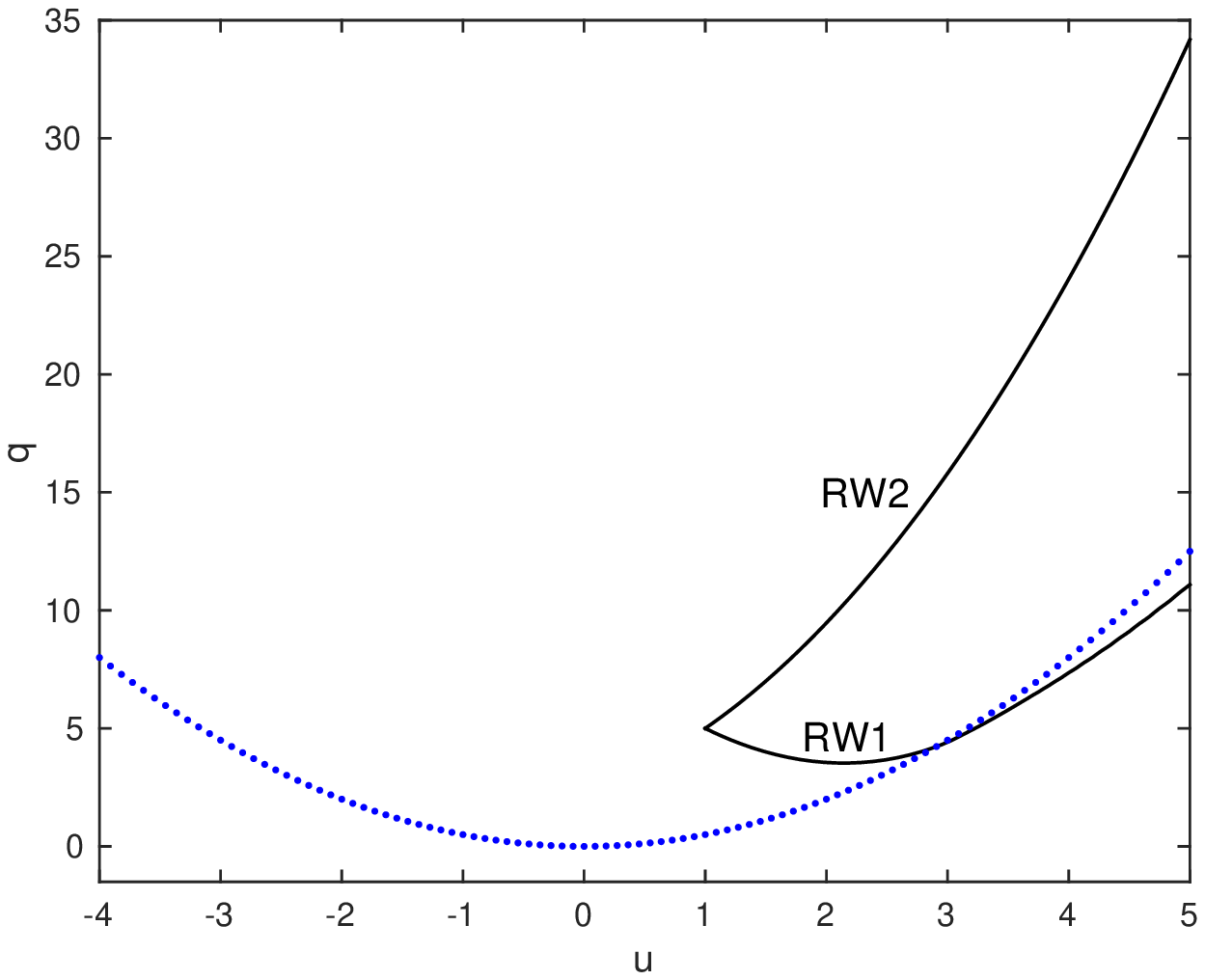}}
     (a)~~~~~~~~~~~~~~~~~~~~~~~~~~~~~~~~~~~~~~~~~~~~~~~~~~~~~~~~(b)
  \end{center}
  \caption{\small (a) Shock wave curves of the first and the second families at the left state $(u_L,q_L)=(1,5)$. 
                  (SW1) is indicated by the upper curve, while (SW2) is the lower curve. 
                  The blue dotted curve shows the limiting curve $q=u^2/2$.
                  (b) Rarefaction wave curves of the first and the second families at the left state $(u_L,q_L)=(1,5)$.
                  The (RW1) is indicated by the lower curve while (RW2) is the upper curve.}
\label{Fig1}
\end{figure}
\begin{flalign}\label{SW1}
& \mathrm{(SW1)} \ \
\nonumber
& q_R=q_L -\frac{1}{2}\big(u_L-u_R\big)\big(2u_R-1\big) 
\qquad \qquad \qquad \qquad \qquad \qquad \qquad \qquad \quad \
\\ 
& & + \mid u_L-u_R\mid \Big( 2q_L + \Sfrac{1}{2}(u_L-u_R)-\Sfrac{1}{3} \big(2u_L^2 +2u_Lu_R-u_R^2 \big) 
+ \Sfrac{1}{4} \Big)^{\frac{1}{2}}, \ \
\end{flalign}
for $u_R<u_L$. To verify that this indeed is the shock wave of the first 
family, we obtain from \eqref{Rankine-Hugoniot} and \eqref{Lax} that 
\begin{equation*}
\lambda_-(u_L,q_L)\geq c=\frac{2u_R-1-\sqrt{8q_L+1+\frac{4u_R^2}{3}-\frac{8u_Lu_R}{3}-\frac{8u_L^2}{3}-2u_R+2u_L}}{2}.
\end{equation*}
 Taking into account the form of $\lambda_-$, we conclude from the above equation that
\begin{equation*}
2(u_L-u_R)\geq \sqrt{8q_L+1-4u_L^2} -\sqrt{8q_L+1+\frac{4u_R^2}{3}-\frac{8u_Lu_R}{3}-\frac{8u_L^2}{3}-2u_R+2u_L}.
\end{equation*}
Further simplification leads to
\begin{equation*}
2 \geq \frac{-\frac{4}{3}(u_L-u_R)-2}{\sqrt{8q_L+1-4u_L^2} +\sqrt{8q_L+1+\frac{4u_R^2}{3}-\frac{8u_Lu_R}{3}-\frac{8u_L^2}{3}-2u_R+2u_L}},
\end{equation*}
which is obviously correct. In a similar way, the second part of the Lax condition,
\begin{equation*}
\lambda_-(u_R,q_R)\leq c,
\end{equation*}
can be verified. Moreover, it is trivial to verify the additional inequality $\lambda_+(u_R,q_R) \geq c$,
so that we have three characteristic curves entering the shock trajectory, and one characteristic curve
leaving the shock.
%
\begin{flalign}
& \mathrm{(SW2)} \ \
\nonumber
& q_R=q_L-\frac{1}{2}\big(u_L-u_R\big)\big(2u_R-1\big)
\qquad \qquad \qquad \qquad \qquad \qquad \qquad \qquad \quad  \ \\ 
& & - \mid u_L-u_R\mid \Big( 2q_L + \Sfrac{1}{2} (u_L-u_R) - \Sfrac{1}{3}\big( 2u_L^2 +2u_Lu_R-u_R^2 \big) +
 \Sfrac{1}{4} \Big)^{\frac{1}{2}}, \ \
\end{flalign}
for $u_R<u_L$. We will skip the proof since it is the same as in the case of (SW1).
Next, we have the rarefaction curves. \\

\noindent
(RW1), \ Using the method from \cite[Theorem 7.6.5]{Daf}, this wave can be written as
\begin{equation}
\label{RW1}
\frac{dq}{du}=\frac{2u-1-\sqrt{8q-4u^2+1}}{2}=\lambda_-(u,q), \ \ \ \ \ q(u_L)=q_L,
\end{equation} 
for $u_R>u_L$. 
Clearly, for $u_R<u_L$ we cannot have (RW1) since in that domain, 
states are connected by (SW1) (see (SW1) above). 
In order to prove that \eqref{RW1} indeed provides RW1, 
we need to show that 
\begin{equation}
\label{check-1}
\lambda_-(u_L,q_L)<\lambda_-(u_R,q_R) \ \ {\rm if} \ \ u_R>u_L.
\end{equation} 
Introducing the change of variables $\tilde{q}=8q-4u^2+1$ in \eqref{RW1}, we can rewrite it in the form
$$
\frac{d\tilde{q}}{du}=-4(1+\sqrt{\tilde{q}})<0.
$$ 
From here, we see that $\tilde{q}$ is decreasing with respect to $u$ and thus, 
for $u_L<u_R$, we must have
$$
8q_L-4u_L^2+1=\tilde{q}_L>\tilde{q}_R=8q_R-4u_R^2+1.
$$ 
This, together with $u_L<u_R$ immediately implies \eqref{check-1}. \\

\noindent
(RW2) Using again \cite[Theorem 7.6.5]{Daf}), we have
\begin{equation}
\label{RW2}
\frac{dq}{du}=\frac{2u-1+\sqrt{8q-4u^2+1}}{2}=\lambda_+(u,q), \ \ q(u_L)=q_L,
\end{equation} 
for $u_R>u_L$. 
It can be shown that \eqref{RW2} gives the rarefaction wave (RW2) in the same way 
explained above for (RW1).
The wave fan issuing from the left state $(u_L,q_L)$ 
and the inverse wave fan issuing from the right state $(u_R,q_R)$ 
are given in Figure \ref{Fig2}(a) and Figure \ref{Fig2}(b), respectively.
\begin{figure}
  \begin{center}
     {\includegraphics[scale=0.5]{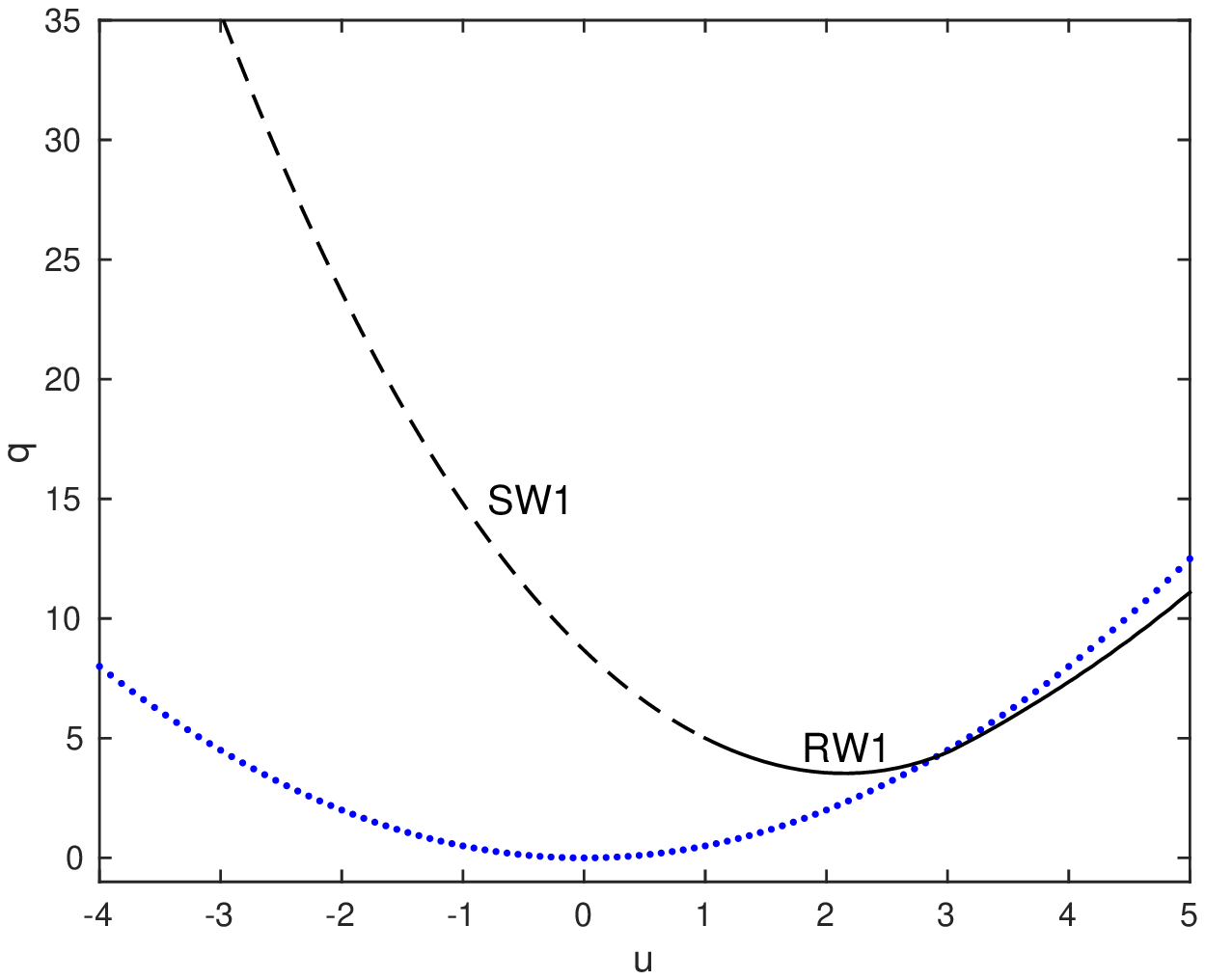}}~~
     {\includegraphics[scale=0.5]{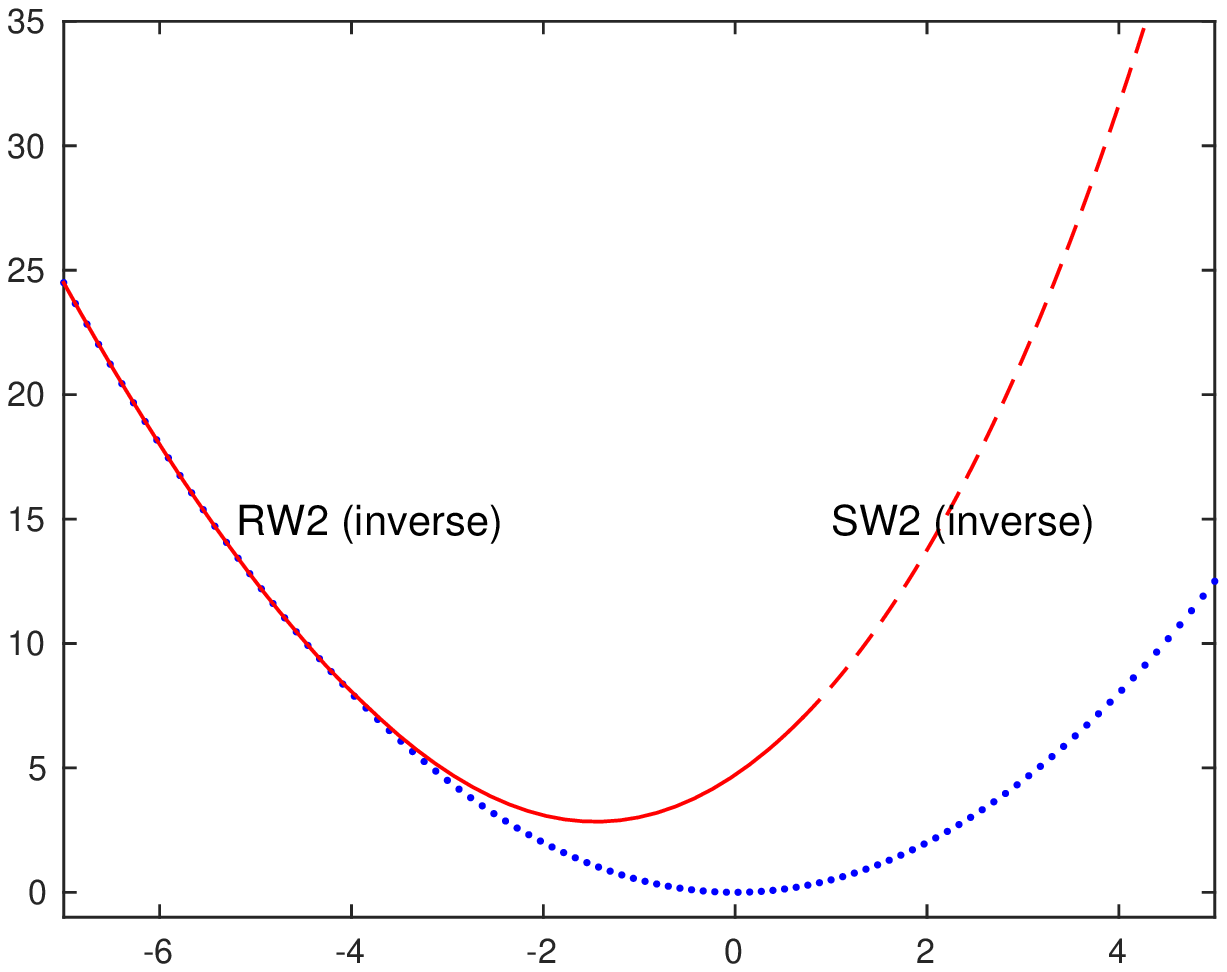}}
     (a)~~~~~~~~~~~~~~~~~~~~~~~~~~~~~~~~~~~~~~~~~~~~~~~~~~~~~~~~(b)
  \end{center}
  \caption{\small Shock and rarefaction wave curves of the first and the second families: 
                        (a) shows SW1 (dashed) and RW1 (solid) at the left state $(u_L,q_L)=(1,5)$.
                        (b) shows inverse SW2 (dashed, red) and inverse RW2 (solid, red) 
                        at the right state $(u_R,q_R)=(0.7,7)$.}
\label{Fig2}
\end{figure}

We next aim to prove existence of solution for arbitrary Riemann initial data 
without necessarily assuming a small enough initial jump. 
The only essential hypothesis is that both left and right states are above the critical curve $q_{crit}= u^2/2$:
\begin{equation}
\label{assumpt}
q_L\geq u_L^2/2, \ \ \ \ q_R \geq u_R^2/2.
\end{equation} 
This assumptions is of course natural given the change of variables $q=\frac{u^2+v^2}{2}$. 
Nevertheless, this condition makes complicates our task 
since is also needs to be shown that the Lax admissible solution 
to a Riemann problem remains in the area $q\geq u^2/2$. 
To this end, the following lemma will be useful.
\begin{lem}
\label{L1}
The function $q_{crit}(u)=\frac{u^2}{2}$ satisfies \eqref{RW2}.
\end{lem}
\begin{proof}
The proof is obvious and we omit it.
\end{proof} 
The above lemma is important since, according to the uniqueness of solutions 
to the Cauchy problem for ordinary differential equations, 
it shows that if the left and right states $(u_L,q_L)$ and $(u_R,q_R)$ 
are above the curve $q_{crit}(u)=\frac{u^2}{2}$, 
then the simple waves (SW1, SW2, RW1, RW2) connecting the states will remain above it 
which means that we can use the solution to \eqref{Transformed system} 
to obtain a solutions of \eqref{Brio System} 
since the square root giving the function $v=\sqrt{2q-u^2}$ will be well defined. 
Concerning the Riemann problem, we have the following theorem.
\begin{thm}
\label{transf-thm}
Given a left state $(u_L,q_L)$ and a right state $(u_R,q_R)$, so that both are above the critical curve 
$q_{crit}(u)=\frac{u^2}{2}$ i.e. we have $q_L\geq u_L^2/2$ and $q_R\geq u_R^2/2$,
the states $(u_L,q_L)$ and $(u_R,q_R)$ can be connected 
Lax admissible shocks and rarefaction waves via a middle state belonging to the domain $q>u^2/2$.
\end{thm}
%
\begin{figure}
  \begin{center}
    {\includegraphics[scale=0.5]{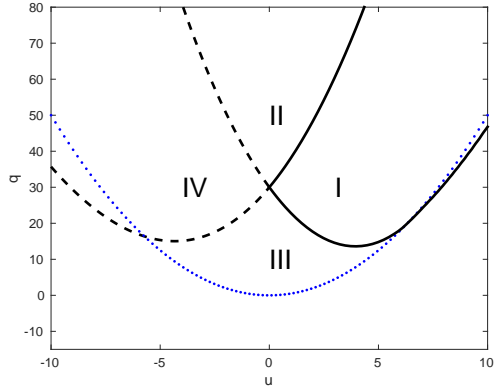}}
  \end{center}
  \caption{\small Admissible connections between a given left state $(u_L,q_L)$ and a right state 
                  can be classified into four regions in the phase plane.}
\label{Fig3}
\end{figure}
\begin{proof}
In order to find a connection between $(u_L,q_L)$ and $(u_R,q_R)$, we first draw the waves of the 
first family (SW1 and RW1) through $(u_L,q_L)$ and waves of the second family (SW2 and RW2) through $(u_R,q_R)$. 
The point of intersection will be the middle state through which we connect $(u_L,q_L)$ and $(u_R,q_R)$ 
(see Figure \ref{Fig4} for different dispositions of $(u_L,q_L)$ and $(u_R,q_R)$). 
In this case, the intersection point will be unique which can be seen by considering 
the four possible dispositions of the states $(u_L,q_L)$ and $(u_R,q_R)$ shown in Figure \ref{Fig4}:
\begin{itemize}

\item
For right states in region $I$: RW1 followed by RW2;

\item
For right states in region $II$: SW1 followed by RW2;

\item
For right states in region $III$: RW1 followed by SW2;

\item
For right states in region $IV$: SW1 followed by SW2;

\end{itemize} 
Properties of the curves of the first and second families are provided in a)-d) above. The growth properties give also existence as we shall show in detail in the sequel of the proof. 

Firstly, we remark that SW1 and RW1 emanating from $(u_L,q_L)$ cover the entire $q\geq u^2/2$ domain (see Figure \ref{Fig2}(a)). In other words, we have for the curve $q_R$ defining the SW1 by \eqref{SW1}:
\begin{equation*}
\lim\limits_{u_R\to - \infty} q(u_R)=\infty,
\end{equation*}
implying that the SW1 will take all $q$-values for $q_R> q_L$. More precisely, for every $q_R>q_L$ there exists $u_R<u_L$ such that $q_R(u_R)=q_R$ where $q_R$ is given by \eqref{SW1}.

%
\begin{figure}
  \begin{center}
     {\includegraphics[scale=0.45]{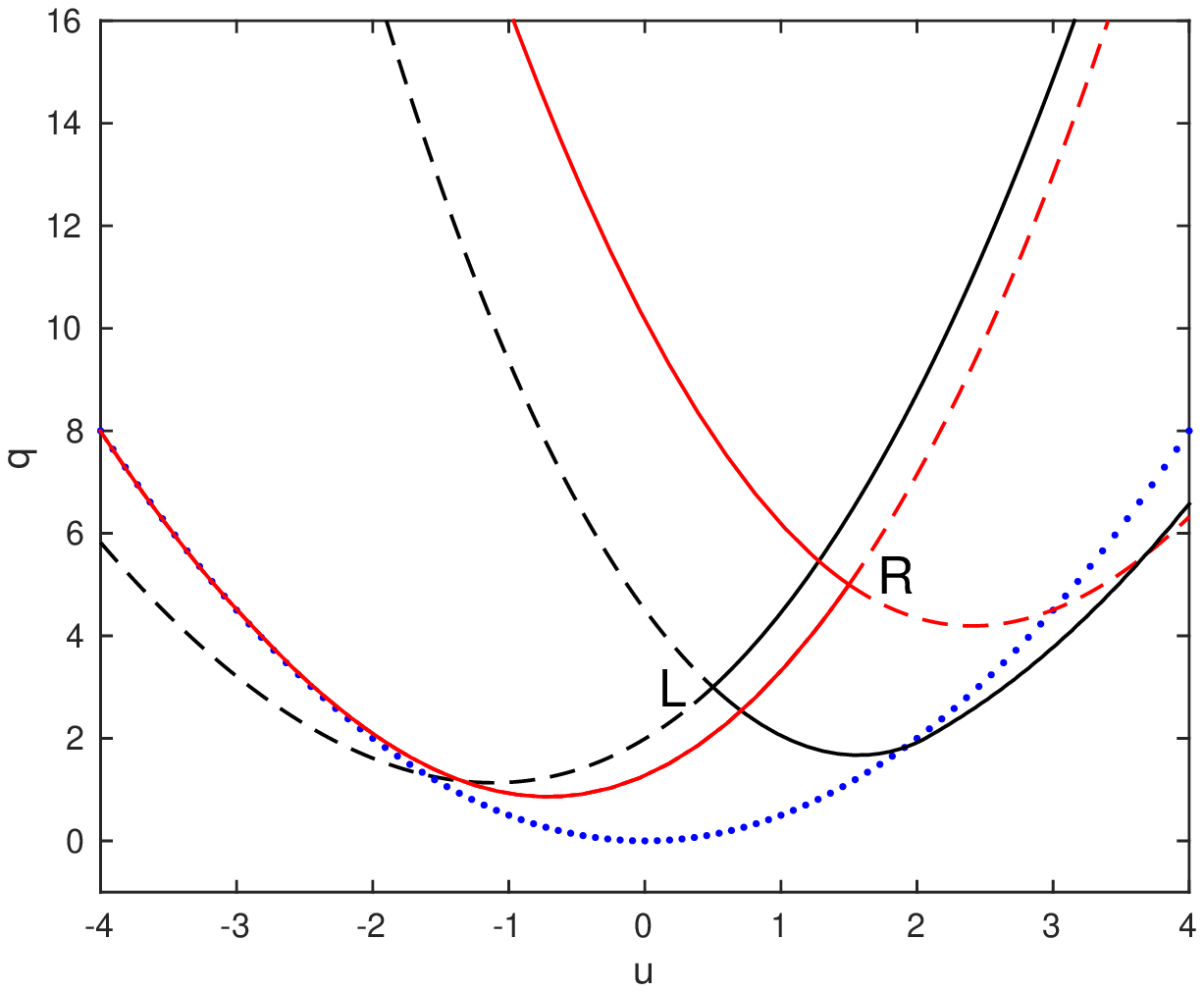}}~~
     {\includegraphics[scale=0.45]{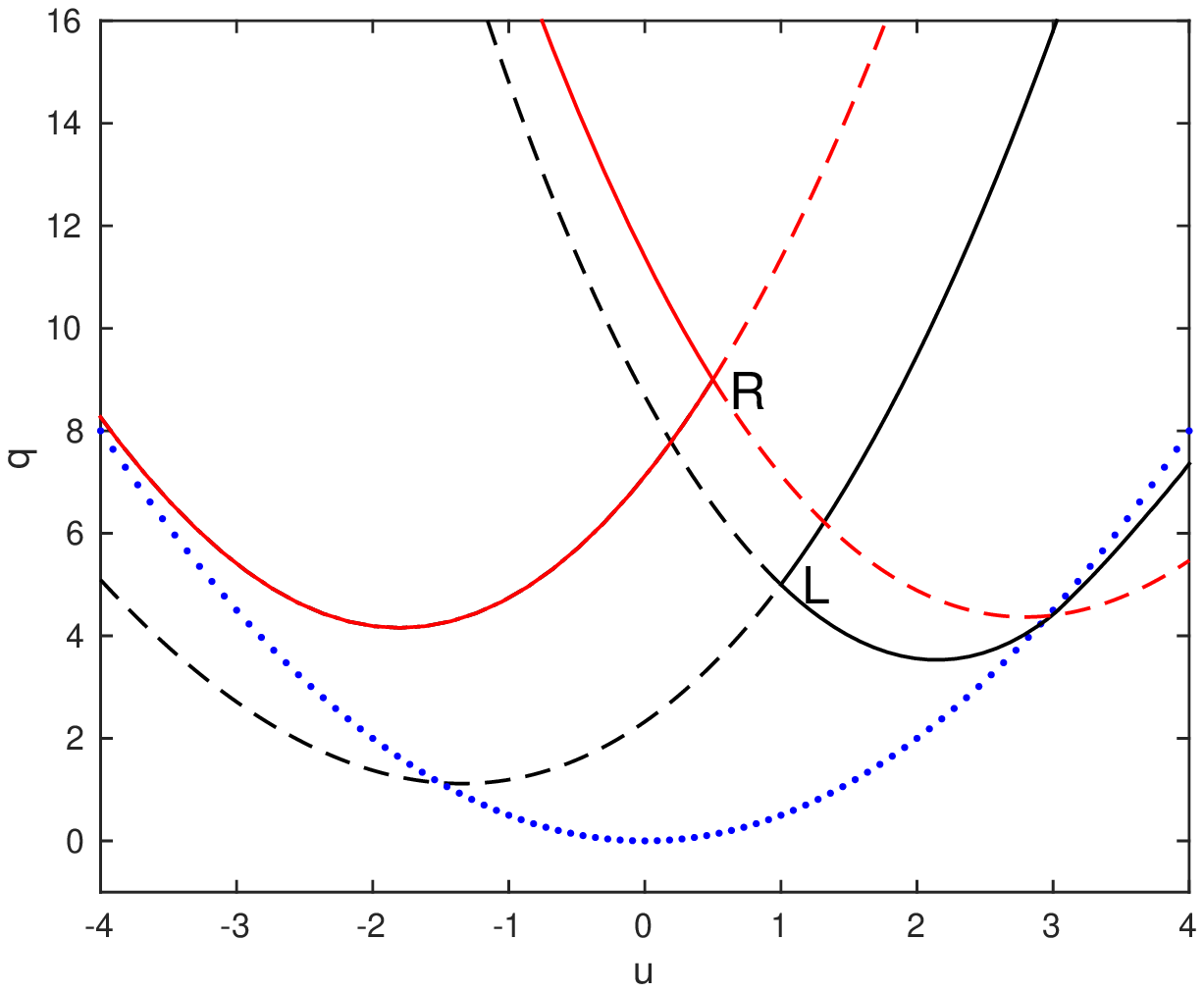}}\\
     (a)~~~~~~~~~~~~~~~~~~~~~~~~~~~~~~~~~~~~~~~~~~~~~~~~~~~~~~~~(b)
      {\includegraphics[scale=0.45]{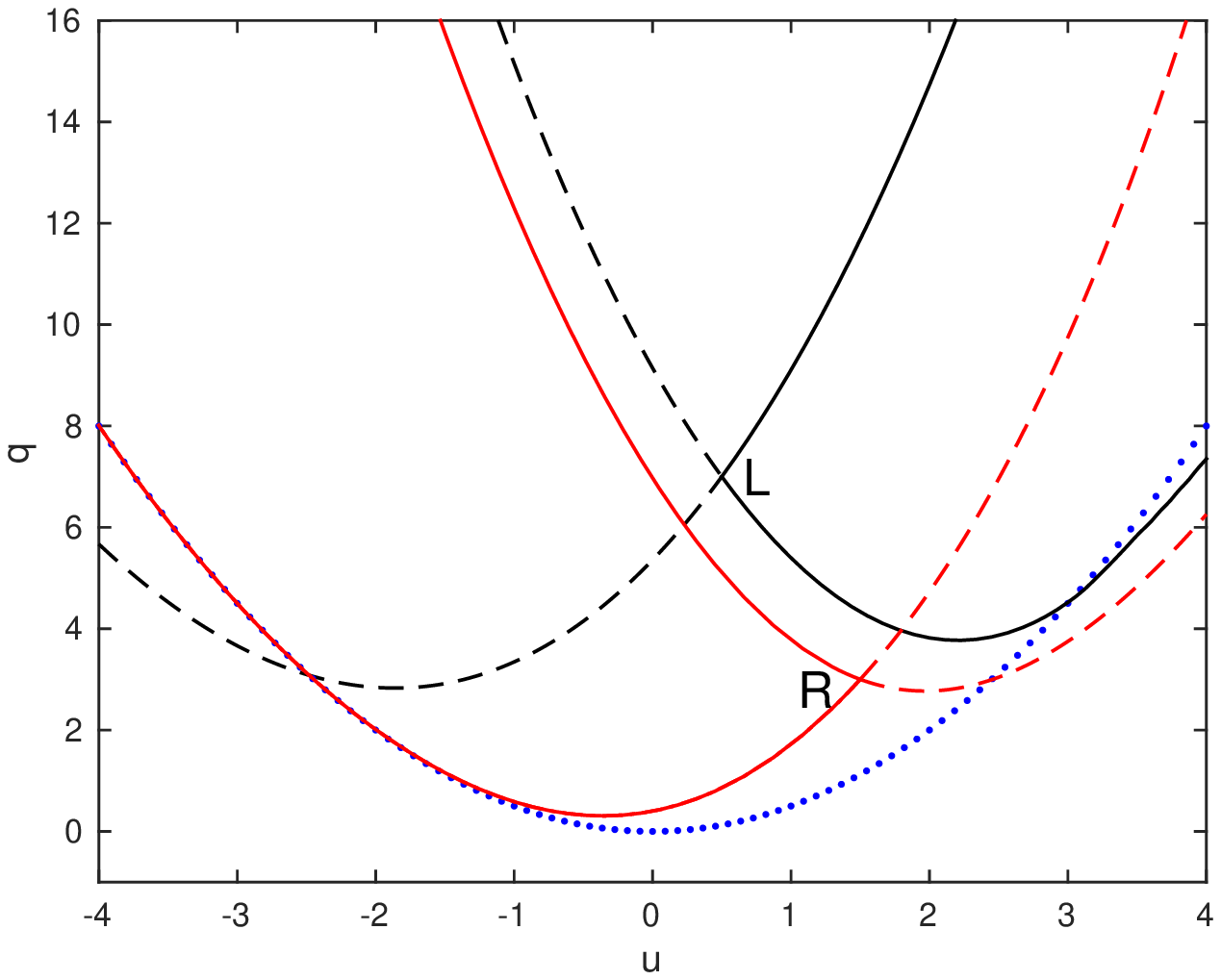}}~~
      {\includegraphics[scale=0.45]{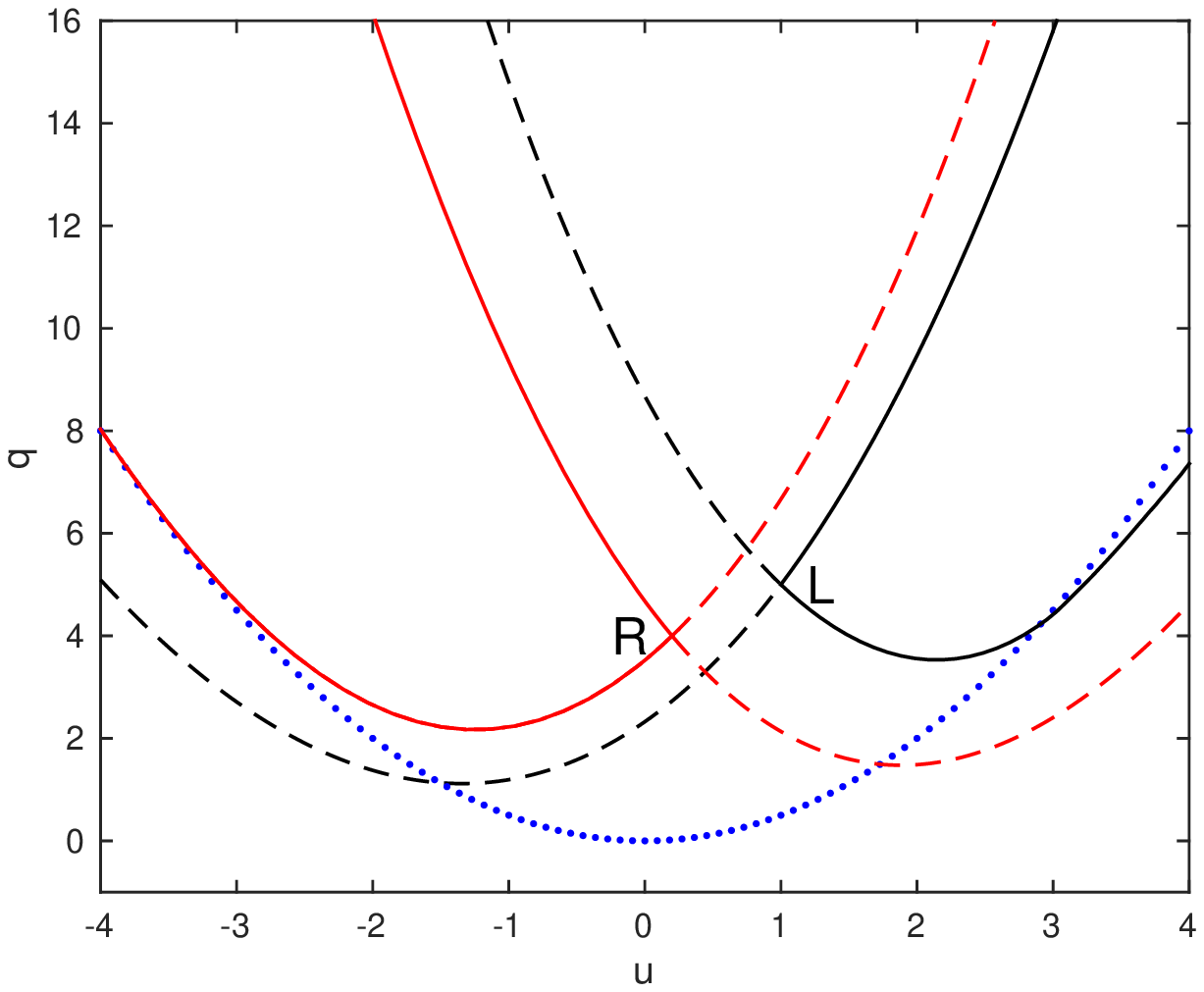}}\\
 
    (c)~~~~~~~~~~~~~~~~~~~~~~~~~~~~~~~~~~~~~~~~~~~~~~~~~~~~~~~~(d)
  \end{center}
    \caption{\small Shock and rarefaction wave curves of the first and the second families. 
                    At the left state $L=(u_L,q_L)$, the curves SW1 (dashed), SW2 (dashed), RW1 (solid), 
                    and RW2 (solid) are drawn in black. The inverse curves at the right state $R=(u_R,q_R)$ 
                    are indicated in red: SW1 (dashed), SW2 (dashed), RW1 (solid) and RW2 (solid).
                    Panel (a) shows the situation for region I,
                    Panel (b) shows the situation for region II,
                    Panel (c) shows the situation for region III and
                    Panel (d) shows the situation for region IV.
}
\label{Fig4}
\end{figure}

As for the RW1, it holds for $q$ given by \eqref{RW1} that 
\begin{equation*}
\frac{dq}{du}-u \leq -1 \implies \frac{dq}{du}\leq u-1,
\end{equation*}
which means that the RW1 curve emanating from any $(u_L,q_L)$ for which $q_L>u_L^2/2$ will intersect the curve $q_{crit}=\frac{u^2}{2}$ (since $\frac{d q_{crit}}{d u}=u>u-1 \geq \frac{dq}{du}$) at some $u_R>u_L$ as shown in Figure \ref{Fig1}, b). 

Now, we turn to the waves of the second family. Let us fix the right state $(u_R,q_R)$. We need to compute the inverse waves (i.e. for the given right state, we need to compute curves consisting of appropriate left states (see Figure \ref{Fig2}(b)). The inverse rarefaction curve of the second family is given by the equation \eqref{RW2}, but we need to take values for $u_R<u_L$ (opposite to the ones given in \eqref{RW2}). As for the inverse SW2, we compute from \eqref{Rankine-Hugoniot} and \eqref{Rankine-Hugoniot1} the value $q_L$:

\begin{align}\label{SW2-inv}
\nonumber
q_L= q_R&-\frac{1}{2} \big(u_L-u_R\big)\big(2u_L-1\big)\\ 
& + \frac{(u_L-u_R)}{2}\sqrt{8q_R+1+\Sfrac{4u_L^2}{3}-\Sfrac{8u_L u_R}{3}-\Sfrac{8u_R^2}{3}-2u_L+2u_R},
\end{align}
for $u_R<u_L$.
Clearly, the RW2 cannot intersect the critical line $q_{crit}=\frac{u^2}{2}$ since $q_{crit}$ satisfy \eqref{RW2} (see Lemma \ref{L1}) and the intersection would contradict uniqueness of solution to the Cauchy problem for \eqref{RW2}. However, a solution to \eqref{RW2} with the initial conditions $q(u_R)=q_R>u_R^2/2$  will converge toward the line $q_{crit}=u^2/2$ since for $q$ given by \eqref{RW2} we have
\begin{equation*}
\frac{dq}{du}-u \geq 0 \ \ {\rm and} \ \ \frac{dq}{du}\Big|_{(u,u^2/2)}-u=0,
\end{equation*}
implying that $q$ will decrease toward $q_{crit}=u^2/2$ and that they will merge as $u_L\to -\infty$ (see Figure \ref{Fig2}(b)). As for the inverse SW2 given by \eqref{SW2-inv}, we see that
\begin{equation*}
\lim\limits_{u_L\to \infty} q(u_L)=\infty,
\end{equation*}
which eventually imply that the 1-wave family emanating from $(u_L,q_L)$ must intersect with the inverse 2-wave family emanating from $(u_R,q_R)$ somewhere in the domain $q>u^2/2$ (see Figure \ref{Fig4} 
for several dispositions of the left and right states).

Finally, we remark that according to the previous analysis, it follows that the intersection between curves of the first and the second family is unique. \end{proof}

\section{Admissibility conditions for $\delta$-shock wave solution to the original Brio System}

Our starting point is that the system original Brio system \eqref{Brio System}
is based on conservation of quantities which are not necessarily physically conserved,
and that the transformed system \eqref{Transformed system} is a closer representation 
of the physical phenomenon to be described. The second principle is that $\delta$-distribution 
represents actually a defect in the model and thus, 
it should be necessarily present as a part of non-regular solutions to \eqref{Brio System}. 
Moreover, the regular part of a solution to \eqref{Brio System} 
should be an admissible solution to \eqref{Transformed system}. 
Having these requirements in mind, we are able to introduce admissibility conditions 
for a $\delta$-type solution to \eqref{Brio System}.

Let us first recall the characteristic speeds for \eqref{Brio System}. Following the \cite{HL}, we see immediately that
\begin{equation}
\label{lambdaB}
\lambda_1(u,v)=u-1/2-\sqrt{v^2+1/4}, \ \ \lambda_2(u,v)=u-1/2+\sqrt{v^2+1/4}.
\end{equation} The shock speed for \eqref{Brio System} for the shock determined by the left state $(U_L,V_L)$ and the right state $(U_R,V_R)$ is given by
\begin{equation}
\label{shockB}
s=\frac{U_L+U_R}{2}+\frac{V_L^2-V_R^2}{2(U_L-U_R)}.
\end{equation}

Now, we can formulate admissibility conditions for $\delta$-type solution to \eqref{Brio System} in the sense of Definition \ref{def-gen}. We shall require that the real part of $\delta$-type solution to \eqref{Brio System} satisfy the energy-velocity conservation system \eqref{Transformed system} and that the number of $\delta$-distributions appearing as part of the solution to \eqref{Brio System} is minimal. 

\begin{df}
\label{admissibility} We say that the pair of distributions
$u=U+\alpha(x,t)\delta(\Gamma)$ and $v=V+\beta(x,t)\delta(\Gamma)$ satisfying Definition \ref{def-gen} with $f(u,v)=\frac{u^2+v^2}{2}$ and $g(u,v)=v(u-1)$ is an admissible $\delta$-type solution to \eqref{Brio System}, \eqref{riemann} if

\begin{itemize}

\item The regular parts of the distributions $u$ and $v$ are such that the functions $U$ and $q=(U^2+V^2)/2$ represent Lax-admissible solutions to \eqref{Transformed system} with the initial data 
\begin{equation}
\label{Triemann}
u|_{t=0}=U_0, \ \ q|_{t=0}= q_0=(U_0^2+V_0^2)/2.
\end{equation}

\item For every $t\geq 0$, the support of the $\delta$-distributions appearing in $u$ and $v$ is of minimal cardinality.

\end{itemize} 
\end{df} To be more precise, the second requirement in the last definition means that the admissible solution will have ``less" $\delta$-distributions as summands in the $\delta$-type solution  than any other $\delta$-type solution to \eqref{Brio System}, \eqref{riemann}. We have the following theorem:

\begin{thm}
There exists a unique admissible  $\delta$-type solution to \eqref{Brio System}, \eqref{riemann}.
\end{thm} 
\begin{proof}
We divide the proof into two cases:

In the first case, we consider initial data such that both left and right states 
of the function $V_0$ have the same sign. In the second case, we consider the initial data where left and right states 
of the function $V_0$ have the opposite sign. 

In the first case, we first solve \eqref{Transformed system} with the initial data $U_0$ and $q_0=(U_0^2+V_0^2)/2$. 
According to Theorem \ref{transf-thm}, there exists a unique Lax admissible solution 
to the problem denoted by $(U,q)$. Using this solution, we define $V=\sqrt{2q-U^2}$ 
if the sign of $V_0$ is positive and  $V=-\sqrt{2q-U^2}$ if the sign of $V_0$ is negative.

To compute $\alpha$ and $\beta$ in \eqref{delta-sol}, we compute the Rankine-Hugoniot deficit 
if it exists at all. According to Theorem \ref{transf-thm} 
there are four possibilities.

\begin{figure}
  \begin{center}
     {\includegraphics[scale=0.8]{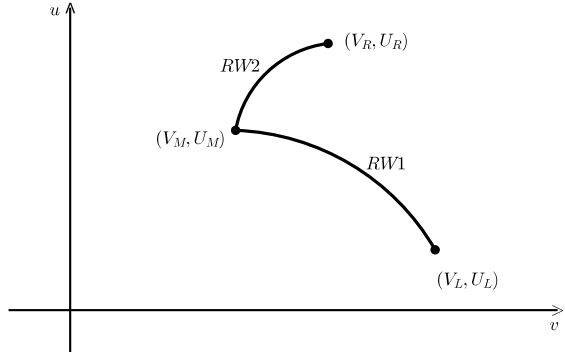}}
  \end{center}
  \caption{\small Admissible connection between rarefaction wave curves of the first and second families}
\label{Fig5}
\end{figure}

\begin{itemize}

\item Region $I$:
The states $(U_L,q_L)$ and $(U_R,q_R)$ are connected by a combination of RW1 and RW2 
via the state $(U_M,q_M)$. In this situation, we do not have any Rankine-Hugoniot deficit 
since the solution $(u,q)$ to \eqref{Transformed system} is continuous. 
Thus, we simply write $(u,v)=(u,\sqrt{2q-u^2})$ 
and this is the solution to \eqref{Brio System}, \eqref{riemann}.  
The solution is plotted in Figure \ref{Fig5}.

As for the uniqueness, we know that the function $u$ is unique 
since it is the Lax admissible solution to \eqref{Transformed system} 
with the initial data \eqref{Triemann}. 
The function $v$ is determined by the unique functions $u$ and $q$ via 
$$
v=\pm \sqrt{2q-u^2}.
$$ 
Thus, $v$ could change sign so that we connect $V_L$ by $V_{M1}$ 
and then skip to $-V_{M1}$ on $v=-\sqrt{2q-u^2}$ and then connect it by $-V_{M2}$. 
From here we connect to $V_{M2}$ located on the original curve $v=\sqrt{2q-u^2}$ 
and then connect $V_{M2}$ to $V_M$. Finally, we connect $V_M$ with $V_R$. 
The procedure is illustrated in Figure \ref{Fig6}. 
However, since we imposed the requirement that the solutions have a minimal number 
of $\delta$-distributions and we cannot connect the states $(U_{M1},V_{M1})$ 
and $(U_{M1},-V_{M1})$ using the $\delta$-shock since such a choice would yield a solutions
with a higher number of singular parts than the previously described solution.

Thus the shock connecting the states $(U_{M1},V_{M1})$ and $(U_{M1},-V_{M1})$ 
cannot be singular, (i.e. there can be no Rankine-Hugoniot deficit),
and therefore the speed $s$ of the shock  must satisfy the Rankine-Hugoniot condition
$$
s=U_{M1}.
$$ 
On the other hand, the characteristic speeds of $(U_{M1},V_{M1})$ and $(U_{M1},-V_{M1})$ 
are $\lambda_1(U_{M1},V_{M1})=\lambda_1(U_{M1},-V_{M1}) \neq s$, and since these are equal,
the shock connection between $(U_{M1},V_{M1})$ and $(U_{M1},-V_{M1})$ 
is impossible with Rankine-Hugoniot condition satisfied.

Similarly, the same requirement makes it impossible to connect $(U_{M2},V_{M2})$ and $(U_{M2},-V_{M2})$ 
by a $\delta$-shock. In this case, the shock speed satisfies the Rankine-Hugoniot condition 
$$
s=U_{M2}.
$$ 
Furthermore, we have equality of speeds $\lambda_2(U_{M2},V_{M2})=\lambda_2(U_{M2},-V_{M2})$, 
but we have the contrasting inequality $\lambda_2(U_{M2},V_{M2})=\lambda_2(U_{M2},-V_{M2})\neq s$ 
implying that a shock connection between $(U_{M2},V_{M2})$ and $(U_{M2},-V_{M2})$ is not possible 
if the Rankine-Hugoniot condition is satisfied. 
The same procedure leads to the conclusion that a $\delta$-shock connection 
between $(U_{M},V_{M})$ and $(U_{M},-V_{M})$ is impossible with Rankine-Hugoniot condition satisfied.

Hence, the only possible connection of $(U_L,V_L)$ and $(U_R,V_R)$ is 
by the combination RW1 and RW2 via the state $(U_M,V_M)$. 
Consequently, we remark that RW1 and RW2 corresponding to \eqref{Transformed system} 
are transformed via $(u,q)\mapsto (u,\sqrt{2q-u^2})$ into RW1 and RW2 corresponding 
to \eqref{Brio System} 
(since $q$ is the entropy function for \eqref{Brio System}, 
and RW1 and RW2 are smooth solutions to \eqref{Transformed system}).

\begin{figure}
  \begin{center}
     {\includegraphics[scale=0.7]{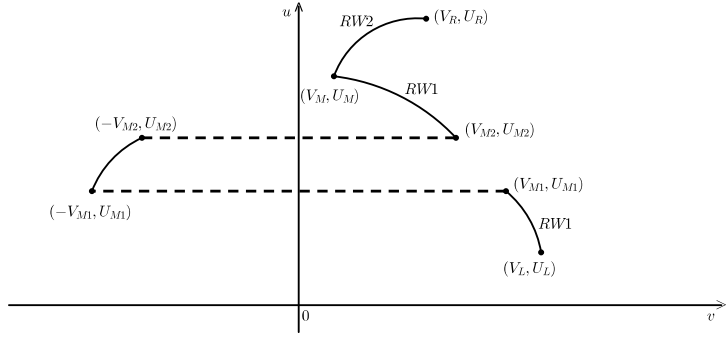}}
  \end{center}
  \caption{\small Nonadmissible connection between rarefaction wave curves of the 					first and the second families}
\label{Fig6}
\end{figure}

\item Region $II$:
The states $(U_L,q_L)$ and $(U_R,q_R)$ are connected by the combination SW1 and RW2 via the state $(U_M,q_M)$.

Unlike the previous case, we have a shock wave in \eqref{Transformed system},
and we will necessarily have a Rankine-Hugoniot deficit in the original system \eqref{Brio System}. 
We thus define 
\begin{equation}
\label{S1R2}
(u,v)=(u,\sqrt{2q-u^2})+(0,\beta(t) \delta(x-ct)),
\end{equation} 
where $c$ is the speed of the SW1 connecting the states $(U_L,q_L)$ and $(U_M,q_M)$ 
in \eqref{Transformed system}. The speed $c$ is given by \eqref{RH-def1} 
as well as the corresponding Rankine-Hugoniot deficit $\beta(t)$:
\begin{equation}
\label{RHdef}
c=\frac{\frac{U_L^2+V_L^2}{2}-\frac{U_R^2+V_R^2}{2}}{U_L-U_R}, \ \ \beta(t)=(c(V_L-V_R)-(V_L(U_L-1)-V_R(U_R-1))t.
\end{equation} 
Concerning the other possible solutions, as in the previous item, 
we can only split the curve connecting $(U_L,V_L)$ and $(U_M,V_M)$ into several new curves 
e.g. by connecting the states $(U_L,V_L)$ and $(U_{M1},V_{M1})$, 
then the (opposite with respect to $v$) states $(U_{M1},V_{M1})$ and $(U_{M1},-V_{M1})$, 
then $(U_{M1},-V_{M1})$ and $(U_{M2},-V_{M2})$, 
then $(U_{M2},-V_{M2})$ and $(U_{M2},V_{M2})$ etc. until we reach $(U_M,V_M)$. 
The states $(U_{M1},V_{M1})$ and $(U_{M1},-V_{M1})$ can be connected only by the shock satisfying 
the Rankine-Hugoniot conditions 
(due to the minimality condition on $\delta$-shocks, we cannot have a Rankine-Hugoniot deficit). 

Since we cannot have the Rankine-Hugoniot deficit, as in the previous item, 
we must connect the various states with shock waves satisfying the Rankine-Hugoniot conditions, 
and at the same time being equal to the speed $c$ 
(the speed of the SW1 connecting the states $(U_L,q_L)$ and $(U_M,q_M)$ in \eqref{Transformed system}). 
This is obviously never fulfilled i.e. the only solution in this case is \eqref{S1R2}.

\item Region $III$:

The states $(U_L,q_L)$ and $(U_R,q_R)$ are connected by the combination RW1 and SW2 via the state $(U_M,q_M)$.

The analysis for the existence and uniqueness proceeds along the same lines as the first two cases. 
The admissible (and thus unique) $\delta$-type solution in this case has the form:
\begin{equation}
\label{R1S2}
(u,v)=\big(u,\sqrt{2q-u^2}\big)+\big(0,\beta(t) \delta(x-ct)\big),
\end{equation}
where $c$ in this case represents the speed of the SW2 connecting 
the states $(U_R,q_R)$ and $(U_M,q_M)$ in \eqref{Transformed system}. 
The speed $c$ and the corresponding Rankine-Hugoniot deficit $\beta(t)$ are given in \eqref{RH-def1} 
and explicitly expressed as in \eqref{RHdef}. The solution structure is represented by 
\begin{equation*}
(U_L,V_L)\xrightarrow{RW1}(U_M,V_M)\xrightarrow{SW2}(U_R,V_R),
\end{equation*}
where the $\delta$-shock propagates at the speed $c=\lambda_1(U_M,V_M)=\lambda_1(U_M,-V_M)$. 
Notice that it is possible to generate infinitely many non-admissible 
(in the sense of Definition \ref{admissibility}) solutions 
(in the sense of Definition \ref{def-gen}) by partitioning the rarefaction wave 
of the first family that connects the states $(U_L,V_L)$ and $(U_M,V_M)$. 
The solution is constructed by connecting $(U_L, V_L)$ and $(U_{M1},V_{M1})$ 
by RW1 and then passing over to $(U_{M1},-V_{M1})$ by a shock 
which satisfies the Rankine-Hugoniot conditions $s=U_{M1}$. 
The procedure is advanced to connect all the finite possible states $(U_{Mk},V_{Mk})$ and $(U_{Mk},-V_{Mk})$ 
by a shock satisfying both the Rankine-Hugoniot conditions $s=U_{Mk}$, where $k\in\mathbb{Z}_+$ 
and the speed of the shock of the second family connecting the states $(U_M,-V_M)$ and $(U_R,V_R)$. 
This process is carried out prior to the state $(U_M,V_M)$ 
and the shocks connecting pairs of states cannot be admissible in the sense of Definition 3.1 
due to the minimality condition. 
Consequently, the only solution admissible in this sense is \eqref{R1S2}.

\item Region $IV$:
The states $(U_L,q_L)$ and $(U_R,q_R)$ are connected by the combination SW1 and SW2 via the state $(U_M,q_M)$.

The presence of shocks in this case will necessarily introduce Rankine-Hugoniot 
deficit in \eqref{Brio System}. The solution is constructed by solving 
\eqref{Transformed system} for the solution $(u,q)$ and then go back to 
\eqref{Brio System} to obtain the admissible $\delta$-type solution 
\begin{equation}\label{S1andS2}
(u,v)=\big(u,\sqrt{2q-u^2}\big)+\big(0,\beta_1(t)\delta(x-c_1t)\big)+\big(0,\beta_2(t)\delta(x-c_2t)\big),
\end{equation}
where $c_1$ and $c_2$ given by the expressions
\begin{equation}
c_1=\frac{\frac{U_L^2+V_L^2}{2} - \frac{U_M^2+V_M^2}{2}}{U_L-U_M} \ \ \ \ \text{and} \ \ \ \ \ c_2=\frac{\frac{U_M^2+V_M^2}{2} - \frac{U_R^2+V_R^2}{2}}{-U_M-U_R},
\end{equation}
are the speeds of the shocks SW1 and SW2 respectively. The Rankine-Hugoniot 
deficits $\beta_1(t)$ and $\beta_2(t)$ are expressed as in \eqref{RHdef} for 
the appropriate states. The analysis for uniqueness of \eqref{S1andS2} is 
similar to the above cases except that all the elementary waves involved in this 
case are shocks.

\end{itemize}

Now, assume that $V_L>0$ and $V_R<0$. It was shown in \cite{HL} that in this case, 
the Riemann problem \eqref{Brio System}, \eqref{riemann} does not admit a Lax admissible solution,
even for initial data with small variation.

In order to get an admissible $\delta$-type solution, as before, 
we solve \eqref{Transformed system} with $(U_0,q_0)$ as the initial data. 
The obtained solution connects $(U_L,q_L)$ with $(U_R,q_R)$ by Lax admissible waves through a 
middle state $(U_M,q_M)$. Next, we go back to the original system \eqref{Brio System} 
by connecting $(U_L,V_L)$ with $(U_M,\sqrt{2q_M-U_M^2})$ by an elementary wave 
containing the corresponding Rankine-Hugoniot deficit corrected by the $\delta$-shock wave. 
Then, we connect $(U_M,\sqrt{2q_M-U_M^2})$ with $(U_M,-\sqrt{2q_M-U_M^2})$ 
by the shock wave whose speed will obviously be $U_M$. 
Finally, we connect $(U_M,-\sqrt{2q_M-U_M^2})$ with $(U_R,V_R)$ 
by an elementary wave containing corresponding Rankine-Hugoniot deficit 
corrected by the $\delta$-shock wave.

Let us first show it is possible to apply the described procedure. 
We again need to split considerations into four possibilities depending 
on how the states $(U_L,q_L)$ and $(U_R,q_R)$ are connected.

\begin{figure}
  \begin{center}
     {\includegraphics[scale=0.8]{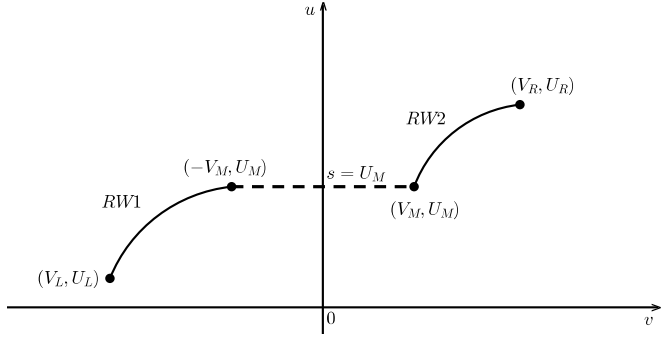}}
  \end{center}
  \caption{\small Admissible connection between rarefaction wave curves of the first and second families
in the case when the left state has $V_L < 0$ and the right state has $V_R > 0$. In this case, a shock
connecting the states $(-V_M,U_M)$ and $(V_M,U_M)$ has to be fitted between the rarefaction curves.
It is shown in the part of the proof pertaining to region I that this shock has the required speed.
}
\label{Fig7}
\end{figure}

\begin{itemize}

\item  Region $I$:
The states $(U_L,q_L)$ and $(U_R,q_R)$ are connected by RW1 and RW2 via the middle state $(U_M,q_M)$.

It is clear that we can connect $(U_L,V_L)$ with $(U_M,\sqrt{2q_M-U_M^2})$ using RW1 
(it is the same for both equations since RW1 and RW2 are smooth solutions to \eqref{Transformed system}). 
Also, we can connect $(U_M,-\sqrt{2q_M-U_M^2})$ with $(U_R,V_R)$ using RW2. 
We need to prove that the shock wave connecting $(U_M,\sqrt{2q_M-U_M^2})$ 
and $(U_M,-\sqrt{2q_M-U_M^2})$ has a speed which is 
between $\lambda_1(U_M,\sqrt{2q_M-U_M^2})$ and $\lambda_2(U_M,-\sqrt{2q_M-U_M^2})$.

In other words, we need to check
\begin{align*}
U_M - \frac{1}{2} - \sqrt{V_M^2 + \Sfrac{1}{4}} \leq U_M \leq U_M - \frac{1}{2} + \sqrt{V_M^2 + \Sfrac{1}{4}}
\end{align*} 
which is obviously correct. This configuration is depicted in Figure \ref{Fig7}.

\item  Region $II$:

The states $(U_L,q_L)$ and $(U_R,q_R)$ are connected by SW1 and SW2 via the middle state $(U_M,q_M)$.

As in the previous item, we connect $(U_L,V_L)$ with $(U_M,\sqrt{2q_M-U_M^2})$ this time using the SW1 from \eqref{Transformed system} which will induce the Rankine-Hugoniot deficit in \eqref{Brio System}. Then, we skip from $(U_M,\sqrt{2q_M-U_M^2})$ to $(U_M,-\sqrt{2q_M-U_M^2})$ using the standard shock wave (the one satisfying the Rankine-Hugoniot conditions), and finally we go from $(U_M,-\sqrt{2q_M-U_M^2})$ to $(U_R,V_R)$ using the SW2 from \eqref{Transformed system} and corrected with an appropriate $\delta$-shock. More precisely, the admissible $\delta$-type solution will have the form:
\begin{equation}
\label{S1S2}
\begin{split}
u(t,x)&=U_L+(U_M-U_L)(H(x-c_1 t)-H(x-ct))\\&+(-U_M-U_L)(H(x-c t)-H(x-c_2t))+(U_R-U_L)H(x-c_2t)\\
v(t,x)&=V_L+(V_M-V_L) (H(x-c_1 t)-H(x-ct))\\&+(V_M-V_L)(H(x-c t)-H(x-c_2t))+(V_R-V_L)H(x-c_2t)\\&+\beta_1(t) \delta(x-c_1t)+\beta_2(t) \delta(x-c_2t),
\end{split}
\end{equation}where $c_1$ is the speed of the SW1 connecting $(U_L,q_L)$ with $(U_M,q_M)$ in \eqref{Transformed system}, $c_2$ is the speed of the SW2 connecting $(U_M,q_M)$ with $(U_R,q_R)$ in \eqref{Transformed system}, while $c$ is the speed of the shock connecting $(U_M,-\sqrt{2q_M-U_M^2})$ with $(U_M,\sqrt{2q_M-U_M^2})$ and it is given by the Rankine-Hugoniot conditions from \eqref{Brio System}. The deficits $\beta_1$ and $\beta_2$ are given by Theorem \ref{thm-cnl} (see \eqref{RHdef} for the analogical situation).

However, we still need to prove that \eqref{S1S2} is well defined, i.e. that
\begin{align*}
&c_1 \leq c \leq c_2 \ \ \implies \\
& \frac{2 U_M-1-\sqrt{8 q_L+1+\frac{4 U_M^2}{3}-\frac{8 U_L U_M}{3}-\frac{8 U_L^2}{3}-2 U_M + 2 U_L}}{2} \leq U_M \\&
\leq \frac{2U_M-1+\sqrt{8q_R+1+\frac{4U_M^2}{3}-\frac{8 U_M U_R}{3}-\frac{8 U_R^2}{3}-2U_M+2U_R}}{2} 
\end{align*} 
which is also clearly true.

\item  Region $III$:
The states $(U_L,q_L)$ and $(U_R,q_R)$ are connected by RW1 and SW2 via the middle state $(U_M,q_M)$.

This case, as well as the following one, is handled by combining the previous two cases.

\item Region $IV$:

The states $(U_L,q_L)$ and $(U_R,q_R)$ are connected by SW1 and RW2 via the middle state $(U_M,q_M)$.

\end{itemize} 

Uniqueness is obtained by arguing as in the first part of the proof.

\end{proof}

\section*{Acknowledgements}
This research was supported by the Research Council of Norway under grant no. 213747/F20
and grant no. 239033/F20.

\end{document}